\documentclass[12pt]{article}

\usepackage{graphicx}
\usepackage{nameref}
\usepackage[shortlabels,inline]{enumitem}
\usepackage{empheq}
\usepackage[shortlabels]{enumitem}
\setlist[enumerate]{nosep}
\usepackage[doc]{optional}
\usepackage{xcolor}
\usepackage[colorlinks=true,
            linkcolor=refkey,
            urlcolor=lblue,
            citecolor=red]{hyperref}
\usepackage{float}
\usepackage{soul}
\usepackage{graphicx}
\definecolor{labelkey}{rgb}{0,0.08,0.45}
\definecolor{refkey}{rgb}{0,0.6,0.0}
\definecolor{Brown}{rgb}{0.45,0.0,0.05}
\definecolor{lime}{rgb}{0.00,0.8,0.0}
\definecolor{lblue}{rgb}{0.5,0.5,0.99}

 \usepackage{mathpazo}

\colorlet{hlcyan}{cyan!30}

\usepackage{stmaryrd}
\usepackage{amssymb}
\makeatletter
\def\namedlabel#1#2{\begingroup
   \def\@currentlabel{#2}%
   \label{#1}\endgroup
}
\makeatother

\newcommand{\sepp}{\setlength{\itemsep}{-2pt}}
\newcommand{\seppthree}{\setlength{\itemsep}{-3pt}}

\newcommand{\R}[1]{\ensuremath{{\operatorname{R}}_%
{#1}}}
\newcommand{\Pj}[1]{\ensuremath{{\operatorname{P}}_%
{#1}}}
\newcommand{\Rt}[1]{\ensuremath{{\operatorname{Rot}}%
({#1})}}
\newcommand{\Rf}[1]{\ensuremath{{\operatorname{Refl}}%
({#1})}}
\newcommand{\Rfop}{\ensuremath{{\operatorname{Refl}}}}
\newcommand{\Rtop}{\ensuremath{{\operatorname{Rot}}}}

\newcommand{\menge}[2]{\big\{{#1}~\big |~{#2}\big\}}

\newcommand{\fenv}[1]%
{\ensuremath{\,\overrightarrow{\operatorname{env}}_{#1}}}
\newcommand{\benv}[1]%
{\ensuremath{\,\overleftarrow{\operatorname{env}}_{#1}}}

\newcommand{\scal}[2]{\left\langle{#1},{#2}  \right\rangle}

\newcommand{\RR}{\ensuremath{\mathbb R}}

\newcommand{\ZZ}{\ensuremath{\mathbf Z}}

\newcommand{\ran}{\ensuremath{{\operatorname{ran}}\,}}

\newcommand{\Fix}{\ensuremath{\operatorname{Fix}}}
\newcommand{\Id}{\ensuremath{\operatorname{Id}}}

{\begin{list}{}{%
\settowidth{\labelwidth}{\textrm{#1~}}%
\setlength{\leftmargin}{\labelwidth+\labelsep}}}
{\end{list}}
\usepackage{amsthm}
\usepackage[capitalize,nameinlink]{cleveref}
\crefname{equation}{}{equations}
\crefname{chapter}{Appendix}{chapters}
\crefname{item}{}{items}
\crefname{enumi}{}{}
\newtheorem{theorem}{Theorem}[section]
\newtheorem{lemma}[theorem]{Lemma}

\newtheorem{proposition}[theorem]{Proposition}

\newtheorem{example}[theorem]{Example}
\newtheorem{ex}[theorem]{Example}
\newtheorem{fact}[theorem]{Fact}
\newtheorem{remark}[theorem]{Remark}

\providecommand{\RR}{\mathbb{R}}

\providecommand{\ran}{\operatorname{ran}}

\providecommand{\Id}{\operatorname{{ Id}}}

\providecommand{\ZZ}{{\bf Z}}

\providecommand{\ran}{\operatorname{ran}}

\providecommand{\Id}{\operatorname{Id}}

\providecommand{\R}{{ R}}

\providecommand{\RR}{\mathbb{R}}

\definecolor{myblue}{rgb}{.8, .8, 1}

\allowdisplaybreaks 

\begin{document}

%

\author{
Salihah Alwadani\thanks{
Mathematics, University
of British Columbia,
Kelowna, B.C.\ V1V~1V7, Canada. E-mail:
\texttt{saliha01@mail.ubc.ca}.},~ 
Heinz H.\ Bauschke\thanks{
Mathematics, University
of British Columbia,
Kelowna, B.C.\ V1V~1V7, Canada. E-mail:
\texttt{heinz.bauschke@ubc.ca}.}
~and~
Xianfu Wang\thanks{
Mathematics, University
of British Columbia,
Kelowna, B.C.\ V1V~1V7, Canada. E-mail:
\texttt{shawn.wang@ubc.ca}.}
}

\title{\textsf{
Fixed points of compositions 
of nonexpansive mappings: \\ 
finitely many linear reflectors\footnote{
Dedicated to Terry Rockafellar on the occasion of his 85th birthday
}
}
}

\date{April 26, 2020}

\maketitle

\begin{abstract}
Nonexpansive mappings play a central role in modern optimization and
monotone operator theory because their fixed points can describe 
solutions to optimization or critical point problems. 
It is known that when the mappings are sufficiently ``nice'',
then 
the fixed point set of the composition coincides with the intersection 
of the individual fixed point sets. 

In this paper, we explore the situation for compositions 
of linear reflectors. We provide positive results, upper bounds, and limiting 
examples. We also discuss classical reflectors in the Euclidean plane. 
\end{abstract}
{ 

}

\section{Introduction}

Throughout, we assume that 
\begin{equation}
\text{$X$ is
a real Hilbert space with inner product
$\scal{\cdot}{\cdot}\colon X\times X\to\RR$, }
\end{equation}
and induced norm $\|\cdot\|\colon X\to\RR\colon x\mapsto \sqrt{\scal{x}{x}}$.
A mapping $R\colon X\to X$ is \emph{nonexpansive} if 
$(\forall x\in X)(\forall y\in X)$
$\|Rx-Ry\|\leq\|x-y\|$. 
Nonexpansive operators play a central role in modern optimization 
because the set of fixed points $\Fix R := \menge{x\in X}{x=Rx}$
often represents solutions to optimization or inclusion problems 
(see, e.g., \cite{BC2017}). 
A central question is the following 
\begin{quote}
{\em 
Given nonexpansive maps $R_1,\ldots,R_m$ on $X$ with 
$\bigcap_{i=1}^m\Fix R_i\neq\varnothing$, 
what can we say about $\Fix (R_mR_{m-1}\cdots R_1)$? 
}
\end{quote}
Clearly, 
\begin{equation}
  \label{e:inclusion}
  \textstyle \bigcap_{i=1}^{m}\Fix R_i \subseteq \Fix (R_mR_{m-1}\cdots R_1).
\end{equation}
Note that one cannot expect equality to hold in \cref{e:inclusion}:
\begin{example}
\label{ex:-Id}
Suppose that $X\neq \{0\}$. Then 
$\Fix(-\Id)\cap\Fix(-\Id)=\{0\}$ while 
$\Fix(-\Id)(-\Id)=\Fix(\Id)= X$.
\end{example}
However, equality in \cref{e:inclusion} does hold 
for ``nice'' nonexpansive maps 
such as averaged mappings (see, e.g., \cite[Corollary~4.51]{BC2017})
or even strongly nonexpansive mappings (see \cite[Lemma~2.1]{BruckReich}). 

In this note, we aim to study $\Fix (R_mR_{m-1}\cdots R_1)$ for certain mappings 
that are not nice but that do have some additional structure. To describe this,
let us denote the \emph{projector} (or nearest point mapping) associated 
with a nonempty closed convex subset $C$ of $X$ by $\Pj{C}$.
The corresponding \emph{reflector} 
\begin{equation}
\R{C} := 2\Pj{C}-\Id
\end{equation}
is known to be nonexpansive (see, e.g., \cite[Corollary~4.18]{BC2017}). 
Note that 
\begin{equation}
\label{e:FixR}
\Fix\R{C} = C
\end{equation}
and that $\R{\{0\}} = -\Id$, so the class of 
reflectors is ``bad'' (see \cref{ex:-Id}). 
We also have 
\begin{equation}
  \label{e:Uperp}
  \R{C} = \Pj{C} - \Pj{C^\perp} \quad 
\text{provided that $C$ is a closed linear subspace of $X$.}
\end{equation}
When $C$ is a hyperplane containing the origin, then 
we shall refer to $\R{C}$ as a \emph{classical} reflector.
Classical reflectors are basic building blocks: 
indeed,
the \emph{Cartan-Dieudonn\'e Theorem}
(see, e.g., \cite[Theorem~8.1]{Gallier} or 
\cite[Section~2.4]{Stillwell}) states that 
every linear isometry on $\RR^n$ is the composition of 
at most $n$ classical reflectors.

A very satisfying result is available for two general linear reflectors:

\begin{fact}
  (See \cite[Proposition~3.6]{postdocs}.) 
\label{f:postdocs}
Let $U_1$ and $U_2$ be closed linear subspaces of $X$. 
Then 
\begin{equation}
  \Fix (\R{U_2}\R{U_1}) = (U_1\cap U_2)\oplus (U_1^\perp\cap U_2^\perp).
\end{equation}
and 
\begin{equation}
  \Pj{U_1}\Fix (\R{U_2}\R{U_1})= U_1\cap U_2. 
\end{equation}
\end{fact}

\cref{f:postdocs} was used in 
\cite{postdocs} to analyze the 
Douglas--Rachford operator $T := \tfrac{1}{2}(\Id+\R{U_2}\R{U_1})$.
Note that $\Fix T = \Fix(\R{U_2}\R{U_1})$! 
It was shown that $\Pj{U_1}T^n \to \Pj{U_1\cap U_2}$ pointwise.
Iterating $T$ is actually a special case of employing 
Rockafellar's proximal point algorithm \cite{Rockppa}. 

We also note that \cref{f:postdocs} provides an alternative explanation of 
\cref{ex:-Id}: indeed, set $U_1=U_2=\{0\}$ in \cref{f:postdocs}.
Then $\R{U_1}=\R{U_2}=-\Id$, $U_1^\perp = U_2^\perp = X$, and 
$\Fix(\R{U_2}\R{U_2}) = (U_1\cap U_2) \oplus (U_1^\perp \cap U_2^\perp) = X$. 

\cref{f:postdocs} 
nurtures the hope that 
there might exist a nice formula
for $\Fix (\R{U_3}\R{U_2}\R{U_1})$ and that
there might be a way to recover $U_1\cap U_2\cap U_3$
by projecting $\Fix (\R{U_3}\R{U_2}\R{U_1})$ suitably.
Unfortunately, this hope was crushed with the following example:

\begin{example}
\label{ex:fran}
(See \cite[Example~2.1]{Fran2014}.)
Suppose that $X=\RR^2$, 
$U_1 = \RR(0,1)$, 
$U_2 = \RR(\sqrt{3},1)$, 
and $U_3 = \RR(-\sqrt{3},1)$.
Then 
$U_1\cap U_2\cap U_3 = \{0\}$,
$x := (-\sqrt{3},-1)\in\Fix(\R{U_3}\R{U_2}\R{U_1})$
yet 
$\{\Pj{U_1}x,\Pj{U_2}x,\Pj{U_3}x\} \cap U_1 \cap U_2 \cap U_3 = \varnothing$. 
The fixed point sets for all six permutations of the reflectors 
are depicted in \cref{fig:fran}. 
\end{example}

\begin{figure}
  \centering
   \includegraphics[width=0.95\linewidth]{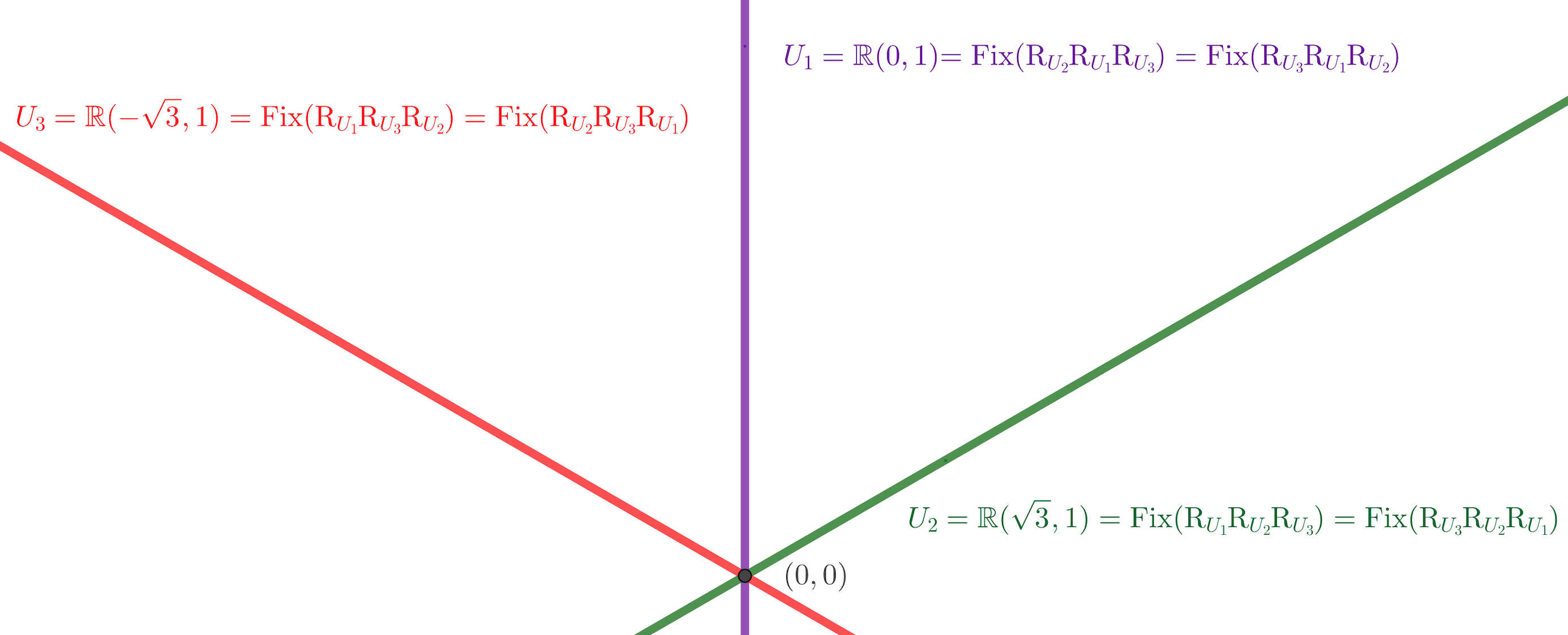}
   \caption{The fixed point sets for \cref{ex:fran}}
   \label{fig:fran}
\end{figure}

We are now in a position to describe precisely our aim. 

{\em The goal of this note is to study the fixed point set of the composition of
finitely many reflectors associated with closed linear subspaces.}

In \cref{sec:positive}, we obtain several positive results 
(see \cref{l:shift} and \cref{t:reversal}), an upper bound 
(see \cref{p:upperbound}) as well as limiting examples.
\cref{sec:plane} focusses mainly on classical reflectors in the Euclidean plane 
for which precise information is available. 
The notation employed is standard and follows largely \cite{BC2017}.

\section{General results}

\label{sec:positive}


We start with a simple observation. 

\begin{lemma}
  \label{l:easy}
Let $U$ be a closed linear subspace of $X$. Then 
\begin{enumerate}
  \item 
  \label{l:easy1}
  $\R{U^\perp}\R{U}=\R{U}\R{U^\perp}=-\Id$
  \item 
  \label{l:easy2}
  $-\R{U}=\R{U}\circ(-\Id)=\R{U^\perp}$
  \item 
  \label{l:easy3}
  $\Fix(-\R{U})=\Fix\R{U^\perp}=U^\perp$.
\end{enumerate}
\end{lemma}
\begin{proof}
We shall employ \cref{e:Uperp} repeatedly. 
\cref{l:easy1}: 
$\R{U^\perp}\R{U} = (\Pj{U^\perp}-\Pj{U})(\Pj{U}-\Pj{U^\perp})
=0 - \Pj{U}-\Pj{U\perp}+0 = -\Id$.
\cref{l:easy2}: 
$-\R{U} = -(\Pj{U}-\Pj{U^\perp}) 
= (\Pj{U^\perp} - \Pj{U^{\perp\perp}}) = \R{U^\perp}$.
\cref{l:easy3}: 
Combine \cref{l:easy2} with \cref{e:FixR}. 
\end{proof}

The following result, which is a consequence of \cref{l:easy}, 
provides a case when 
we have precise knowledge of the fixed point set of 
the composition of three reflectors:

\begin{proposition}
\label{ex:0421a}
Let $U$ and $V$ be closed linear subspaces of $X$. 
Then 
$\R{V}\R{U^\perp}\R{U}=\R{V}\R{U}\R{U^\perp}
=\R{U^\perp}\R{U}\R{V}=\R{U}\R{U^\perp}\R{V}=-\R{V}=\R{V^\perp}$
and thus 
$\Fix(\R{V}\R{U^\perp}\R{U})=\Fix(\R{V}\R{U}\R{U^\perp})
=\Fix(\R{U^\perp}\R{U}\R{V})=\Fix(\R{U}\R{U^\perp}\R{V})=V^\perp$.
\end{proposition}

We now turn to $m$ operators and obtain a very general result which 
clearly shows the effect of cyclically shifting a composition:

\begin{lemma}
\label{l:shift}
Let $R_1,\ldots,R_m$ be arbitrary maps from $X$ to $X$.
Then
\begin{subequations}
\label{e:shift}
\begin{align}
  \Fix(R_mR_{m-1}\cdots R_2R_1) 
  &=
  (R_mR_{m-1}\cdots R_3R_2)\big(\Fix(R_1R_mR_{m-1}\cdots R_3R_2)\big) \\
  &=
  (R_mR_{m-1}\cdots R_3)\big(\Fix(R_2R_1R_mR_{m-1}\cdots R_4R_3)\big) \\
  & \;\,\, \vdots \\
  &=
  R_mR_{m-1}\big(\Fix(R_{m-2}R_{m-3}\cdots R_2R_1R_mR_{m-1})\big)\\
  &=
  R_m\big(\Fix(R_{m-1}R_{m-2}\cdots R_2R_1R_m)\big).
\end{align}
\end{subequations}
\end{lemma}
\begin{proof}
Let $x \in \Fix(R_mR_{m-1}\cdots R_2R_1)$.
Then
$R_1x = R_1(R_mR_{m-1}\cdots R_2R_1)x
= (R_1R_mR_{m-1}\cdots R_3R_2)(R_1x)$
and so 
$R_1x\in \Fix(R_1R_mR_{m-1}\cdots R_3R_2)$. 
It follows that 
\begin{equation}
  R_1\big(\Fix(R_mR_{m-1}\cdots R_2R_1)\big) 
  \subseteq 
  \Fix(R_1R_mR_{m-1}\cdots R_3R_2).
\end{equation}
The same reasoning gives
\begin{subequations}
\begin{align}
  (R_2R_1)\big(\Fix(R_mR_{m-1}\cdots R_2R_1)\big) 
  &\subseteq 
  R_2\big(\Fix(R_1R_mR_{m-1}\cdots R_3R_2)\big)\\
  &\subseteq 
  \Fix(R_2R_1R_mR_{m-1}\cdots R_4R_3),
\end{align}
\end{subequations}
hence 
\begin{subequations}
\begin{align}
  (R_3R_2R_1)\big(\Fix(R_mR_{m-1}\cdots R_2R_1)\big) 
  &\subseteq 
  (R_3R_2)\big(\Fix(R_1R_mR_{m-1}\cdots R_3R_2)\big)\\
  &\subseteq 
  R_3\big(\Fix(R_2R_1R_mR_{m-1}\cdots R_4R_3)\big)\\
  &\subseteq 
  \Fix(R_3R_2R_1R_mR_{m-1}\cdots R_5R_4)
\end{align}
\end{subequations}
until finally 
\begin{subequations}
\label{e:shiftproof}
\begin{align}
  \Fix(R_mR_{m-1}\cdots R_2R_1) 
  &= (R_mR_{m-1}\cdots R_2R_1)\big(\Fix(R_mR_{m-1}\cdots R_2R_1)\big) \\
  &\subseteq 
  (R_mR_{m-1}\cdots R_2)\big(\Fix(R_1R_mR_{m-1}\cdots R_3R_2)\big) \\
  & \;\,\, \vdots \\
  &\subseteq 
  R_m\big(\Fix(R_{m-1}R_{m-2}\cdots R_2R_1R_m)\big)\\
  &\subseteq \Fix(R_mR_{m-1}\cdots R_2R_1).
\end{align}
\end{subequations}
Hence equality holds throughout \eqref{e:shiftproof} 
and we are done.
\end{proof}

\cref{l:shift} allows us to derive a result complementary to \cref{ex:0421a}:

\begin{proposition}
\label{ex:0421b}
Let $U$ and $V$ be closed linear subspaces of $X$. 
Then 
\begin{equation}
\Fix(\R{U^\perp}\R{V}\R{U})=\Fix(\R{U}\R{V}\R{U^\perp})
= \R{U}(V^\perp).
\end{equation}
\end{proposition}
\begin{proof}
Using \cref{l:shift} and \cref{ex:0421a}, we obtain 
\begin{equation}
\Fix(\R{U}\R{V}\R{U^\perp})
=\R{U}(\Fix(\R{V}\R{U}\R{U^\perp}))
=\R{U}(V^\perp).
\end{equation}
Now $\R{U}(V^\perp)$ is a subspace and thus
$\R{U}(V^\perp)
= (-\R{U})(V^\perp)
= \R{U^\perp}(V^\perp)
=\Fix(\R{U^\perp}\R{V}\R{U})$
by the first part of the proof. 
\end{proof}

\begin{remark}
\label{r:ordermatters}
Comparing \cref{ex:0421a} and \cref{ex:0421b},
we note that 
it is
\emph{not necessarily true} 
that 
$\R{U}(V^\perp) = V^\perp$; 
indeed, see 
\cref{ex:hui} below 
for a concrete instance. 
Hence, unlike the case of just two linear reflectors 
(see \cref{f:postdocs}), 
\emph{the order of the operators does influence the fixed point set!}
\end{remark}

While \cref{r:ordermatters} points out the importance of the order 
of the operators, there does exist a nice permutation of the reflectors yielding 
the same fixed point set. 
To describe this result, observe first
$(\R{U_m}\cdots \R{U_1})^* = 
\R{U_1}^*\cdots \R{U_m}^*
=\R{U_1}\cdots \R{U_m}$ 
because linear projectors and (hence) reflectors are self-adjoint. 
Combining this with an old result by Riesz and Sz.-Nagy which states that 
$\Fix T = \Fix T^*$ for any nonexpansive linear operator $T\colon X\to X$
(see \cite[page~408 in Section~X.144]{RSNbook}
or \cite{RSNpaper}),
 we obtain the following positive result:

\begin{theorem}
\label{t:reversal}
Let $U_1,\ldots,U_m$ be closed linear subspaces of $X$. 
Then 
\begin{equation}
\Fix(\R{U_m}\R{U_{m-1}}\cdots \R{U_2}\R{U_1})
= \Fix(\R{U_1}\R{U_{2}}\cdots \R{U_{m-1}}\R{U_m}).
\end{equation}
\end{theorem}

We now turn to three linear subspaces. 
The next result narrows down the location of fixed points.

\begin{theorem}
\label{p:upperbound}
Let $U,V,W$ be closed linear subspaces of $X$. 
Then 
\begin{equation}
  \label{e:0422a}
\Fix(\R{W}\R{V}\R{U})
= \Fix\big(4\Pj{W}\Pj{V}\Pj{U}-2(\Pj{W}\Pj{V}+\Pj{W}\Pj{U}+\Pj{V}\Pj{U})
+\Pj{W}+\Pj{V}+\Pj{U}\big)
\end{equation}
and 
\begin{equation}
\Fix(\R{W}\R{V}\R{U})\subseteq U+V+W. 
\end{equation}
\end{theorem}
\begin{proof}
Let $x\in X$.
Then $x\in \Fix(\R{W}\R{V}\R{U})$
$\Leftrightarrow$
$x=(2\Pj{W}-\Id)(2\Pj{V}-\Id)(2\Pj{U}-\Id)x$
and \cref{e:0422a} follows by expanding and simplifying.
In turn, 
$\Fix(\R{W}\R{V}\R{U})\subseteq U+V+W$ because 
$\Fix(4\Pj{W}\Pj{V}\Pj{U}-2(\Pj{W}\Pj{V}+\Pj{W}\Pj{U}+\Pj{V}\Pj{U})
+\Pj{W}+\Pj{V}+\Pj{U}) 
\subseteq 
\ran(4\Pj{W}\Pj{V}\Pj{U}-2(\Pj{W}\Pj{V}+\Pj{W}\Pj{U}+\Pj{V}\Pj{U})
+\Pj{W}+\Pj{V}+\Pj{U})
\subseteq W+V+U$. 
\end{proof}

The approach utilized in the proof of 
\cref{p:upperbound} to derive the description of the fixed point set 
also works for any \emph{odd} number of reflectors; however, the resulting 
algebraic expressions don't seem to provide further insights.
The superset obtained; however, will easily generalize to 
an \emph{odd} number of reflectors:

\begin{theorem}
\label{p:odd}
Let 
$U_1,\ldots,U_m$ be an \emph{odd} number of closed linear subspaces of $X$. 
Then 
\begin{equation}
\Fix(\R{U_m}\R{U_{m-1}}\cdots \R{U_1})\subseteq U_1+U_2+\cdots + U_m.
\end{equation}
\end{theorem}
\begin{proof}
Let $x\in X$.
Then 
$x \in \Fix(\R{U_m}\cdots \R{U_1})$
$\Leftrightarrow$
$x=\R{U_m}\cdots \R{U_1}x$
$\Leftrightarrow$
$x=(2\Pj{U_m}-\Id)\cdots(2\Pj{U_1}-\Id)x$
$\Rightarrow$
$x\in (-1)^mx + \ran(\sum_{i}\Pj{U_i})$
$\Rightarrow$
$2x \in \ran(\sum_{i}\Pj{U_i}) \subseteq \sum_{i}U_i$.
\end{proof}

\begin{remark}
\cref{p:odd} is false when $m$ is assumed to be even:
indeed, 
assume that $U$ is a proper closed linear subspace of $X$, 
and set $U_1:=U_2:=U$. 
Then 
$\R{U_2}\R{U_1}=\Id$ and hence 
\begin{equation}
\Fix(\R{U_2}\R{U_1})=X \supsetneqq U=U_1+U_2.
\end{equation}
\end{remark}

The next example shows that the upper bound provided in 
\cref{p:odd} is sometimes sharp:

\begin{example}
\label{ex:sharp}
Let 
$U_1,\ldots,U_m$ be closed linear subspaces of $X$  which 
are assumed to be 
\emph{pairwise orthogonal}:
$U_i\perp U_j$ whenever $i\neq j$. 
Then either 
\begin{equation}
\label{ex:sharpodd}
 \text{$m$ is odd and~} \Fix(\R{U_m}\R{U_{m-1}}\cdots \R{U_1})= U_1+U_2+\cdots + U_m
\end{equation}
or 
\begin{equation}
\label{ex:sharpeven}
 \text{$m$ is even and~} 
 \Fix(-\R{U_m}\R{U_{m-1}}\cdots \R{U_1})= U_1+U_2+\cdots + U_m. 
\end{equation}
\end{example}
\begin{proof}
Assume first that $m$ is odd. 
Let $(x_1,\ldots,x_m)\in U_1\times \cdots\times U_m$ 
and set $x=x_1+\cdots + x_m$.
Write $\R{U_i} = \Pj{U_i}-\Pj{U_j^\perp}$ for each $i$ 
(see \cref{e:Uperp}). 
Then
\begin{equation}
\R{U_1}x = (\Pj{U_1}-\Pj{U_1^\perp})(x_1+\cdots+x_m)
= x_1-x_2-\cdots -x_m.
\end{equation}
Hence 
\begin{equation}
\R{U_2}\R{U_1}x = (\Pj{U_2}-\Pj{U_2^\perp})(x_1-x_2\cdots-x_m)
= -x_1-x_2+x_3\cdots +x_m.
\end{equation}
and further 
\begin{equation}
\R{U_3}\R{U_2}\R{U_1}x = (\Pj{U_3}-\Pj{U_3^\perp})(-x_1-x_2+x_3\cdots+x_m)
= x_1+x_2+x_3-x_4-\cdots -x_m.
\end{equation}
In general, for $1\leq k\leq m$, we have 
\begin{equation}
(-1)^{k-1}\R{U_k}\R{U_{k-1}}\cdots\R{U_1}x = (x_1+\cdots+x_k)-(x_{k+1}+\cdots+x_m).
\end{equation}
In particular, because $m-1$ is even, we obtain 
$\R{U_m}\R{U_{m-1}}\cdots \R{U_1}=x$. 
This completes the proof of \cref{ex:sharpodd}.

Now assume that $m$ is even. 
Set $U_{m+1} := \{0\}$.
Then $\R{U_{m+1}}=-\Id$ and $m+1$ is odd.
Therefore, we obtain \cref{ex:sharpeven} from 
the odd case we just proved. 
\end{proof}

In contrast to \cref{ex:sharp}, 
we conclude this section with another example which will 
illustrate that the upper bound in \cref{p:odd} is not always attained:

\begin{example}
Assume that $U$ a closed linear subspace of $X$ such that 
$\{0\}\subsetneqq U$. Let $m$ be an odd positive integer, and
let $i\in\{1,\ldots,m\}$. Then 
set $U_i:=\{0\}$ and $U_j:=U$ for every $j\neq i$.
Then $m-1$ is even and  $\R{\{0\}}=-\Id$. 
Hence 
$\R{U_m}\R{U_{m-1}}\cdots \R{U_1}= -\R{U}^{m-1}
= -\Id$ 
and therefore
\begin{equation}
\Fix(\R{U_m}\R{U_{m-1}}\cdots \R{U_1})=
\Fix(-\Id) = \{0\}\subsetneqq U = U_1+U_2+\cdots+U_m.
\end{equation}
\end{example}

\section{The Euclidean plane $\RR^2$}

\label{sec:plane}

Let us now specialize the general result of the last section 
to the Euclidean plane and classical reflectors. 
We start with some is well known results whose statements can 
be found, e.g., in \cite{WikiRotRef}. 

Set 
\begin{equation}
\Rfop\colon \RR\to \RR^{2\times 2}\colon 
\alpha \mapsto
\begin{pmatrix}
\cos(2\alpha) & \sin(2\alpha)\\
\sin(2\alpha) & -\cos(2\alpha)
\end{pmatrix}.
\end{equation}
It is clear that $\Rfop$ is periodic, with minimal period $\pi$. 
The importance of $\Rfop$ stems from the fact that it 
describes all classical reflectors on $\RR^2$; indeed, 
\begin{equation}
\label{e:0421a}
\R{\RR\cdot(\cos(\alpha),\sin(\alpha))} = \Rf{\alpha}
\end{equation}
for every $\alpha\in\RR$. 
It is convenient to also define 
\begin{equation}
\Rtop\colon \RR\to \RR^{2\times 2}\colon 
\alpha \mapsto
\begin{pmatrix}
\cos(\alpha) & -\sin(\alpha)\\
\sin(\alpha) & \cos(\alpha)
\end{pmatrix}.
\end{equation}
Note that for every $\alpha\in\RR$, 
$\Rt{\alpha}$ describes the counterclockwise rotation 
by $\alpha$; the operator $\Rtop$ is periodic with minimal period $2\pi$. 

The following result provides ``calculus rules'' 
for the composition of reflectors and rotators. It can be verified using 
matrix multiplication and addition theorems 
for sine and cosine. 

\begin{fact}
\label{f:0420}
Let $\alpha$ and $\beta$ be in $\RR$. 
Then the following hold: 
  \sepp
\begin{enumerate}
\item 
\label{f:0420i}
$\Rt{\beta}\Rt{\alpha} = \Rt{\alpha+\beta}$. 
\item 
\label{f:0420ii}
$\Rf{\beta}\Rf{\alpha} = \Rt{2(\beta-\alpha)}$.
\item 
\label{f:0420iii}
$\Rt{\beta}\Rf{\alpha} = \Rf{\alpha+\tfrac{1}{2}\beta}$.
\item 
\label{f:0420iv}
$\Rf{\beta}\Rt{\alpha} = \Rf{\beta-\tfrac{1}{2}\alpha}$. 
\end{enumerate}
\end{fact}

We are now in a position to classify the fixed point sets of 
compositions of classical reflectors on $\RR^2$: 

\begin{theorem}
\label{t:finrefls}
Let $\alpha_1,\ldots,\alpha_m$ be in $\RR$. 
Consider the composition of $m$ classical reflectors,
\begin{equation}
  S_m := \Rf{\alpha_m}\cdots\Rf{\alpha_1}\Rf{\alpha_1},
\end{equation}
and set 
$\beta_m:=\alpha_m-\alpha_{m-1}\pm \cdots -(-1)^m\alpha_1$. 
Then exactly one of the following holds:
\begin{enumerate}
\setlength{\itemsep}{+5pt}
  \item $m$ is odd, 
  $S_m = \Rf{\beta_m}$, and 
  $\Fix S_m = \RR(\cos(\beta_m),\sin(\beta_m))$.
  \item $m$ is even, 
  $S_m = \Rt{2\beta_m}$,
  and 
  $\displaystyle \Fix S_m = \begin{cases}
    \RR^2, &\text{if $\beta_m\in\ZZ\pi$;}\\
    \{0\}, &\text{otherwise.}
  \end{cases}
  $
\end{enumerate}
\end{theorem}
\begin{proof}
We proceed by induction on $m$, discussing the odd and even cases separately.

\emph{Base case:} \emph{Case~1:} Assume that $m=1$. 
Then $\beta_1=\alpha_1$ and 
$S_1= \Rf{\alpha_1} = \Rf{\beta_1}$ so $\Fix S_1 = \Fix \Rf{\alpha_1} =
\RR(\cos(\alpha_1),\sin(\alpha_1)) = \RR(\cos(\beta_1),\sin(\beta_1))$ 
by \cref{e:0421a} as announced. 
\emph{Case~2:} Now assume that $m=2$.
Then $\beta_2= \alpha_2-\alpha_1$. 
Using \cref{f:0420}\ref{f:0420ii}, we obtain 
$S_2 = \Rf{\alpha_2}\Rf{\alpha_1} = \Rt{2(\alpha_2-\alpha_1)}
=\Rt{2\beta_2}$ and the claim follows. 

\emph{Inductive step:}
We assume that the result is true for some integer $m\geq 2$. 
Then 
\begin{equation}
S_{m+1} = \Rf{\alpha_{m+1}}\Rf{\alpha_m}\cdots \Rf{\alpha_2}\Rf{\alpha_1} = 
\Rf{\alpha_{m+1}}S_m
\end{equation}
and 
\begin{equation}
\label{e:0421b}
\beta_{m+1} = \alpha_{m+1}-\beta_{m}.
\end{equation}

\emph{Case~1:} $m+1$ is odd; equivalently, $m$ is even. 
Using the inductive hypothesis, 
\cref{f:0420}\cref{f:0420iv},
and \cref{e:0421b}, we obtain 
\begin{subequations}
\begin{align}
S_{m+1} &= \Rf{\alpha_{m+1}}S_m = \Rf{\alpha_{m+1}}\Rt{2\beta_m}\\
&=\Rf{\alpha_{m+1}-\tfrac{1}{2}(2\beta_m)}
=\Rf{\beta_{m+1}}
\end{align}
\end{subequations}
and the result follows.

\emph{Case~2:} $m+1$ is even; equivalently, $m$ is odd. 
Using the inductive hypothesis, \cref{f:0420}\cref{f:0420ii},
and \cref{e:0421b}, we obtain 
\begin{subequations}
\begin{align}
S_{m+1} &= \Rf{\alpha_{m+1}}S_m = \Rf{\alpha_{m+1}}\Rf{\beta_m}\\
&=\Rt{2(\alpha_{m+1}-\beta_m)}
=\Rt{2\beta_{m+1}}
\end{align}
\end{subequations}
and the result follows.
\end{proof}

The next example, which was used in an algorithmic context in 
\cite[Example~2.30]{1908.11576}, 
illustrates 
\cref{ex:0421a}, \cref{ex:0421b}, and 
\cref{t:finrefls}.

\begin{example}
\label{ex:hui}
(See also \cite[Example~2.30]{1908.11576})
Set $U:=\RR(1,0)$ so that $U^\perp=\RR(0,1)$,
and $V := \RR(1,1)$. 
Then
$\R{U}=\Rf{0}$,
$\R{U^\perp}=\Rf{\pi/2}$, 
and
$\R{V}=\Rf{\pi/4}$. 
Moreover, the following hold:
\begin{enumerate}
  \item 
  $\Fix(\R{U^\perp}\R{V}\R{U})
  = \Fix(\Rf{\pi/2}\Rf{\pi/4}\Rf{0})
  = \RR(\cos(\pi/4),\sin(\pi/4))
  = \RR(1,1) = V = \R{U}(V^\perp)
  $.
  \item 
  $\Fix(\R{U}\R{V}\R{U^\perp})
  = \Fix(\Rf{0}\Rf{\pi/4}\Rf{\pi/2})
  = \RR(\cos(\pi/4),\sin(\pi/4))
  = \RR(1,1) = V = \R{U}(V^\perp)
  $.
  \item 
  $\Fix(\R{V}\R{U^\perp}\R{U})
  = \Fix(\Rf{\pi/4}\Rf{\pi/2}\Rf{0})
  = \RR(\cos(-\pi/4),\sin(-\pi/4))
  = \RR(1,-1) = V^\perp$.
  \item 
  $\Fix(\R{V}\R{U}\R{U^\perp})
  = \Fix(\Rf{\pi/4}\Rf{0}\Rf{\pi/2})
  = \RR(\cos(3\pi/4),\sin(3\pi/4))
  = \RR(-1,1) = V^\perp$.
  \item 
  $\Fix(\R{U^\perp}\R{U}\R{V})
  = \Fix(\Rf{\pi/2}\Rf{0}\Rf{\pi/4})
  = \RR(\cos(3\pi/4),\sin(3\pi/4))
  = \RR(-1,1) = V^\perp$.
  \item 
  $\Fix(\R{U}\R{U^\perp}\R{V})
  = \Fix(\Rf{0}\Rf{\pi/2}\Rf{\pi/4})
  = \RR(\cos(-\pi/4),\sin(-\pi/4))
  = \RR(1,-1) = V^\perp$.
\end{enumerate}
\end{example}

\begin{figure}
  \centering
   \includegraphics[width=0.95\linewidth]{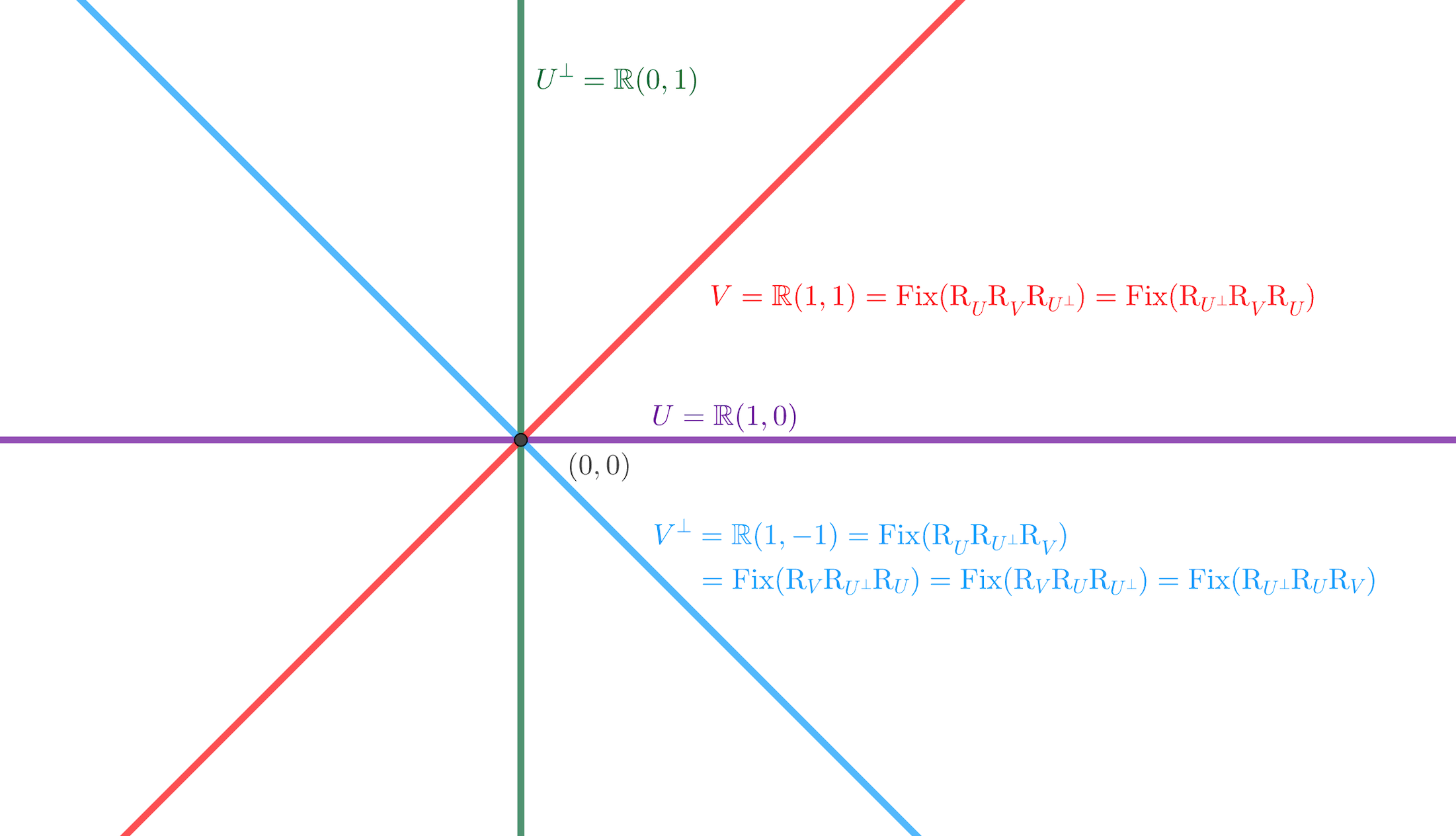}
   \caption{The fixed point sets for \cref{ex:hui}}
   \label{fig:hui}
\end{figure}

\begin{remark}
\cref{ex:hui} clearly shows that the order of the reflectors
influences the fixed point set. 
See also \cref{fig:hui} for a visualization. 
\end{remark}

\begin{example}
\label{ex:superfran}
Let $\gamma\in\RR$,  and 
let $\varepsilon_1,\varepsilon_2,\varepsilon_3$
all be small in absolute value.
Set
$\alpha_1 := \gamma + \pi/6+\varepsilon_1$,
$\alpha_2 := \gamma +\varepsilon_2$,
$\alpha_3 := \gamma -\pi/6+\varepsilon_3$,
and 
$\varepsilon := \varepsilon_1-\varepsilon_2+\varepsilon_3$, and
suppose that 
$U_i = \RR(\cos(\alpha_i),\sin(\alpha_i))$
for $i\in\{1,2,3\}$. 
Then it follows from \cref{t:finrefls} that 
\begin{subequations}
\begin{align}
\R{U_3}\R{U_2}\R{U_1}
&= \Rf{\alpha_3}\Rf{\alpha_2}\Rf{\alpha_1}
= \Rf{\alpha_3-\alpha_2+\alpha_1}\\
&= \Rf{\gamma+\varepsilon}
\end{align}
\end{subequations}
and 
\begin{equation}
\Fix(\R{U_3}\R{U_2}\R{U_1})
=\RR\big(\cos(\gamma+\varepsilon),\sin(\gamma+\varepsilon)\big).
\end{equation}
However, $U_1\cap U_2\cap U_3 = \{0\}$. 
\end{example}

\begin{remark}
Consider the setting of \cref{ex:superfran}.
\begin{enumerate}
  \item No matter which of the operators in 
  $\{\Pj{U_1},\Pj{U_2},\Pj{U_3},\Pj{U_1^\perp},\Pj{U_2^\perp},\Pj{U_3^\perp}\}$
  we apply to $\Fix(\R{U_3}\R{U_2}\R{U_1})$, we always obtain a line and 
  never the singleton 
  $U_1\cap U_2\cap U_3 = \{0\}$.
  \item If each $\varepsilon_i=0$, then 
  $\varepsilon = 0$ and $\Fix(\R{U_3}\R{U_2}\R{U_1})=U_2$.
  If additionally $\gamma=\pi/6$, then 
  $U_2 = \RR(\cos(\pi/6),\sin(\pi/6))
  =\RR(\sqrt{3}/2,1/2)$ 
  and we recover precisely \cref{ex:fran}.
\end{enumerate}
\end{remark}

We conclude with a comment on higher-dimensional Euclidean space.

\begin{remark}[$\RR^3$ and beyond]
Considering reflectors and rotations 
in $\RR^3$ (see, e.g., \cite{Wiki3Drot}) or even $\RR^n$ is more complicated because 
there is no ``easy''
counterpart of \cref{f:0420}. 
However, using the fact that eigenvalues of isometries are always drawn from 
$\pm 1$ or from nonreal complex conjugate pairs of magnitude $1$, 
one obtains at least the parity result that 
\begin{equation}
 m +  n  \equiv \dim(\Fix(R_mR_{m-1}\cdots R_1)) \mod 2
\end{equation}
for $m$ classical reflectors $R_1,\ldots,R_m$ on $\RR^n$. 
\end{remark}

\section*{Acknowledgements}
The research of HHB and XW was partially supported by Discovery Grants
of the Natural Sciences and Engineering Research Council of
Canada.


\end{document}